\documentclass{amsart}
\usepackage{amsmath}
\usepackage{enumerate}
\usepackage{graphicx}

 \newtheorem{thm}{Theorem}[section]
 
 \newtheorem{lem}[thm]{Lemma}
 \newtheorem{prop}[thm]{Proposition}
 \theoremstyle{definition}
 \newtheorem{defn}[thm]{Definition}
 \theoremstyle{remark}
 \newtheorem{rem}[thm]{Remark}
 
 \numberwithin{equation}{section}

\title[Nodal Solutions for Dirac equations with singular nonlinearities]{Existence of nodal solutions for Dirac equations with singular nonlinearities}

\author[L. Le Treust]{Lo\"{i}c Le Treust }

\address{CEREMADE (UMR CNRS no. 7534)\\
Universit\'{e} Paris-Dauphine\\
Place De Lattre De Tassigny\\
F75775 Paris Cedex 16 France}

\email{letreust@ceremade.dauphine.fr}



\begin{document}
\begin{abstract}
We prove, by a shooting method, the existence of infinitely many solutions of the form $\psi(x^0,x) = e^{-i\Omega x^0}\chi(x)$ of the nonlinear Dirac equation
\begin{equation*}
	i\underset{\mu=0}{\overset{3}{\sum}} \gamma^\mu \partial_\mu \psi- m\psi - F(\overline{\psi}\psi)\psi = 0
\end{equation*}
where $\Omega>m>0,$ $\chi$ is compactly supported and 
\[
	F(x) = \left\{\begin{array}{ll}
		p|x|^{p-1}	& \text{if}~ |x|>0\\
		0		& \text{if}~ x=0 
	\end{array}\right.
\]
with $p\in(0,1),$ under some restrictions on the parameters $p$ and $\Omega.$ We study also the  behavior of the solutions as $p$ tends to zero to establish the link between these equations and the M.I.T. bag model ones.
\end{abstract}

\maketitle

\section{Introduction}
We study, in this paper, a relativistic model proposed by Mathieu and Saly \cite{mathieu1984,mathieu1985} that accounts for the internal structure of hadrons, that is how strong interaction forces bind quarks together. Their model and the M.I.T. bag one (see \cite{johnson} and the references therein) have been introduced to approximate the quantum chromodynamics model and to get the confinement of the quarks.

We will look for localized solutions of the nonlinear Dirac equation:
\begin{equation}\label{equation:stati}
i\underset{\mu=0}{\overset{3}{\sum}} \gamma^\mu \partial_\mu \psi- m\psi - F(\overline{\psi}\psi)\psi = 0.
\end{equation}
The notations are the followings: $m>0,$ $\psi :\mathbb{R}^4\rightarrow \mathbb{C}^4$, $\partial_\mu \psi =  \frac{\partial}{\partial x^\mu}$, $0\leq \mu \leq 3$, where we used Einstein's convention for summation over $\mu.$ We write $\overline{\psi}\psi = (\gamma^0\psi,\psi)$ where $(.,.)$ is the usual scalar product and $\gamma^\mu$ are the $4\times4$ Pauli-Dirac matrices \cite{Thaller1992}:
\[\gamma^0 = \left(\begin{array}{cc}I & 0\\ 0&-I\end{array}\right) \mbox{~and~} \gamma^k = \left(\begin{array}{cc}0 & \sigma^k\\ -\sigma^k&0\end{array}\right)\mbox{ for } k=1,2,3,\]
with
\[
	\sigma^1 = \left(\begin{array}{cc}0 & 1\\ 1&0\end{array}\right), \sigma^2 = \left(\begin{array}{cc}0 & -i\\ i&0\end{array}\right), \sigma^3 = \left(\begin{array}{cc}1 & 0\\ 0&-1\end{array}\right).
\]
The function $F:\mathbb{R}\rightarrow \mathbb{R}$ is defined by
\[
	F(x) = \left\{\begin{array}{ll} p|x|^{p-1}& \mbox{if~} |x|>0\\ 0& \mbox{if~} x=0\end{array}\right.
\] 
with $p\in(0,1).$ The solutions are sought among stationary states 
\begin{equation}\label{solutionstationnaire}
	\psi(x^0,x) = e^{-i\Omega x^0}\chi(x)
\end{equation}
where $x = (x^1,x^2,x^3)\in\mathbb{R}^3,$ $\Omega>m$ and  $\chi$ is solution of:
\begin{equation}\label{equation:stati2}
	i\underset{k = 1}{\overset{3}{\sum}} \gamma^k \partial_k \chi +\Omega \gamma^0 \chi- m\chi - 	F(\overline{\chi}\chi)\chi = 0.
\end{equation}
Following \cite{mathieu1984,mathieu1985}, we will search the solutions among functions of the form:
\begin{equation}\label{soler}
\chi(x)=\left(\begin{array}{c}
v(r)\left(\begin{array}{c}1\\0\end{array}\right)\\
iu(r)\left(\begin{array}{c}\cos\theta \\ \sin\theta e^{i\Phi}\end{array}\right)\\
\end{array}\right)
\end{equation}
where  $(r,\theta,\Phi)$ are the spherical coordinates of $x$ in $\mathbb{R}^3$ and $\chi$ is localized, that is :
\[
	\underset{r\rightarrow \infty}{\lim}(u,v) = 0.
\] 
Equation (\ref{equation:stati2}) then becomes a non-autonomous system of ordinary differential equations
\begin{equation}  \label{equation:uv} 
\left\{ \begin{array}{lcl}
  u' + \displaystyle\frac{2u}{r} &= & v(-F(v^2-u^2)-(m-\Omega)) \\
  v'  & = & u(-F(v^2-u^2)-(m+\Omega)).
\end{array} \right.
\end{equation}
Following Mathieu and Saly \cite{mathieu1984,mathieu1985}, we assume that $u$ is  zero at zero and we consider the following Cauchy problem for $x\in \mathbb{R}^+$:
\begin{equation}\label{systequation:uv}
\left\{ \begin{array}{l}
	(\ref{equation:uv})\\
	(u(0),v(0)) = (0,x).
	\end{array}\right.
\end{equation}
We can choose $x$ nonnegative without loss of generality thanks to the symmetry of the equations. For the sake of notation simplicity, we will not write the $p$ dependence unless it is necessary. For instance, we write $F$, $(u,v)$, \eqref{systequation:uv},\dots  \;instead of $F_p,$ $(u_p,v_p)$, \eqref{systequation:uv}$_p,$ \dots

Equation \eqref{equation:stati} has been introduced by  Mathieu and Saly \cite{mathieu1984,mathieu1985} to model the confinement of the relativistic quarks. Their model is called the fractional bag model. They observed numerically that the solutions are compactly supported. 

Balabane, Cazenave and Vazquez \cite{balabane1990} proved rigorously the existence of a ground state for a more general class of nonlinearities $F$ by a shooting method. Moreover, they obtained a necessary and sufficient condition on $F$ for the ground state solution to be compactly supported. The shooting method has already been used to get infinitely many solutions of a nonlinear Dirac equation in a regular setting by Balabane, Cazenave, Douady and Merle \cite{BCDM88} (see also the references therein). 

The main problems we have to face here occur on the set $\{|u|=|v|\}$ because the nonlinearity $F$ is singular at $0$. Since, Balabane, Cazenave and Vazquez \cite{balabane1990} studied the ground state problem,  the trajectories of the solutions they found do not cross this set. Nevertheless, in this paper, we have to consider solutions of this type to get infinitely many solutions. 

Thus, we have to weaken the notion of solution since the Cauchy problem
\begin{equation}\label{systemequation:diag}
\left\{\begin{array}{ll}
	\eqref{equation:uv}\\
	u(R)=v(R)=x
\end{array}\right.
\end{equation}
has no regular solution defined  in a neighborhood of $R$ for $R>0$ and $x\ne 0.$ 
\begin{defn} \label{def:extended}
	Let $0\leq R<R'.$ A function $w\in C^0(R,R')$ is a solution of a system of ordinary differential equations (E) in the extended sense if there exist at most a finite number $n$ of real number $R<R_1<\dots<R_n<R'$ such that $w$ is of class $\mathcal{C}^1$ on $(R,R')\backslash \{R_1,\dots,R_n\}$ and satisfies the equations of system (E) on $(R,R')\backslash \{R_1,\dots,R_n\}.$
\end{defn}

From now on, we will consider solution of this type (see also  \cite{CL55}). Notice that the nonlinearity $F$ allows the zero function to be solution of \eqref{equation:stati}. Thanks to definition \ref{def:extended}, we can thus extend by zero all the solutions which hit zero.

Since we want to use a shooting method, local existence and uniqueness are very important points. But, the main O.D.E. theorems \cite{CL55} fail to show local uniqueness for problem \eqref{systemequation:diag} and existence is not a trivial point. To overcome this, we have to introduce a regularized problem whose solutions satisfy some key qualitative properties similar to the ones of the solutions of the original system of equations \eqref{equation:uv}. The idea consists in introducing an approach system which is hamiltonian near the set $\{|u|=|v|\}$ so that we get local existence and uniqueness. Nevertheless, the solutions of the regularized problem are singular and they are only solutions in the extended sense of definition \ref{def:extended}.

Once this regularization is done, we can adapt  to our framework the shooting method of Balabane, Dolbeault and Ounaies \cite{balabane2003} which established the existence of infinitely many compactly supported solutions for a sub-linear elliptic equation with any given number of nodes. The problems given by the lack of regularity of the nonlinearity in zero occur when the solutions of their system of equations hit zero. Here, these difficulties arise on the bigger set $\{|u|=|v|\}$. Indeed, our main contribution is to deal with the shooting method of \cite{balabane2003} in this singular framework. 

Mathieu \cite{mathieu1985} has already found numerical excited state solutions. But, in this paper, we provide the first rigorous proof of their existence under some restrictions on $p$ and $\Omega.$  Mathieu and Saly \cite{mathieu1984} have also derived relations between these solutions and the M.I.T. bag model ones. Here, we prove rigorously that the ground state solutions of the fractional models converge to the ground state solution of the M.I.T. bag model as $p$ tends to $0$. Nevertheless, we also show that the limits of the sequence of the excited state solutions are not solutions of the M.I.T. bag model equations.

Let us now state our results:
\begin{thm}\label{resultshooting}
There are $\overline{p}\in(0,1)$ and for every $p \in (0,\overline{p})$, a constant $\Omega_p>m$ such that if $\ \Omega> \Omega_p$, there exists an unbounded increasing sequence $(x_k)_{k\in\mathbb{N}}$ of initial data such that for any $k\in \mathbb{N}$,  the Cauchy problem $\displaystyle{(\ref{systequation:uv})}$ has a compactly supported solution which crosses the set $\{(u,0)|u\ne 0\}$ exactly $k$ times. 
\end{thm}
The following theorem establishes the close link between the fractional model and the M.I.T bag one.
\begin{thm}\label{theo:mitlimit}
There is $\overline{\Omega}>m$ and for $\Omega>\overline{\Omega},$ for $k\in\mathbb{N}$,  a finite number of points $R^1,\dots,R^l$ with $l\leq 2k+1$, $(u_0,v_0)\in \mathcal{C}^1(\mathbb{R^+}\backslash\{R^1,\dots,R^l\})\cap L^{\infty}(\mathbb{R}^+)$ and a decreasing sequence $(p_n)$ converging to zero such that :
\begin{enumerate}[(i)]
	\item \label{theoextra:nodal} $v^2_0-u^2_0$ is continuous on $\mathbb{R}^+$,  positive on $[0,R^1)$ and on exactly $k$ intervals $(R^i,R^{i+1}),$ 
	\item \label{theoextra:convergence} $(u_{p_n},v_{p_n})$ converges to $(u_0,v_0)$ uniformly on every compact interval of $\{|v_0^2-u^2_0|>0\}$.
	\item \label{theoextra:diraclibre} $(u_0,v_0)$ is a solution of the free Dirac equation on $[0,R^l]\backslash \{R^1,\dots,R^l\}$:
		\[\left\{\begin{array}{rl}
			u'+\frac{2u}{r}=&v(\Omega-m)\\
			v'=&-u(\Omega+m),
		\end{array}\right.\] discontinuous in $R^1,\dots,R^l$.
	\item \label{theoextra:annulation} $(u_0,v_0)\equiv 0$ on $[R^l,\infty)$ and $(v_0^2-u_0^2)(R^i)=0,$
	
	\end{enumerate}
	where $(u_p,v_p)$ is  the solution of (\ref{systequation:uv}$)_p$ found by theorem \ref{resultshooting} with $k$ nodes.
\end{thm}
Let us notice that $(u_0,v_0)$ is discontinuous at each bound of the $k$ intervals of point (\ref{theoextra:nodal}).  In the case $k=0$, the solution $(u_0,v_0)$ is the ground state of the M.I.T. bag model as Mathieu and Saly derived in \cite{mathieu1984}. Nevertheless, the other nodal solutions that we get, are different from those derived by Mathieu in \cite{mathieu1985} by lack of continuity.

In section \ref{section:preliminaryresult}, we define the hamiltonian regularization.  We will prove that the qualitative properties we need do not depend on the regularization parameter and that the solutions of the regularized system of equations locally exist and are unique. In section \ref{section:shoot}, we prove the existence of compactly supported solutions to the regularized problem by the shooting method.  We finish the proof of theorem \ref{resultshooting} in section \ref{section:theo}. Finally, we study the relation between the fractional model and the M.I.T. bag one in section \ref{section:mit}.
%
%
\section{Notations and preliminary results} \label{section:preliminaryresult}

In this section, we fix $p\in(0,1/2)$ and $\Omega>m.$ Following \cite{balabane1990}, we define the continuous functions:
\[
	H:\mathbb{R}^2\rightarrow\mathbb{R}~\text{and}~H_\epsilon:\mathbb{R}^2\rightarrow\mathbb{R}
\]
by
\begin{equation*}
	H(u,v) = -\frac{1}{2}(v^2-u^2)| v^2-u^2|^{p-1} +\frac{\Omega-m}{2}v^2+\frac{\Omega+m}{2}u^2
\end{equation*}
and
\begin{align*}
	\lefteqn{H_{\epsilon}(u,v) 	= H(u,v) - (\Omega -\epsilon)v^2}\\
	&= -(v^2-u^2)\left(\frac{1}{2}~|v^2-u^2|^{p-1} +\frac{\Omega+m}{2}\right) + \epsilon v^2
\end{align*}
for $\epsilon \in [0,m)$ and $(u,v)\in\mathbb{R}^2.$

These functions will be of constant use in this paper. In the following lemma, we study some of their properties.
\begin{lem}
\label{energielevel}
Let $\epsilon \in [0,m)$, we have that:
\begin{enumerate}[(i)] 
	\item there exist two positive constants $A$ and $B$ such that:
			\[H(u,v) \geq A(u^2+v^2)-B \mbox{~for all~} (u,v) \in \mathbb{R}^2,\]
	\item the set $H^{-1}(-\infty,C)$ is bounded for $C\in \mathbb{R},$
	\item $H^{-1}_{\epsilon}(\{0\})$ is a connected unbounded set in $\{|u|<|v|\}\cup\{0\}$, such that 		\[H^{-1}_{\epsilon}(\{0\})\cap H^{-1}(\{0\}) = \{0\},\]
		 for every $\epsilon >0$,
	\item $H^{-1}_{0}(\{0\}) = \{|u|=|v|\},$
	\item there are two functions $\gamma\mapsto C_\gamma$ and $\gamma\mapsto D_\gamma$ defined for $\gamma\in (0,+\infty)$ such that for all $(u,v) \in \mathbb{R}^2$, $\gamma>0$:
			\[H(u,v) \geq \gamma \Rightarrow \left\{
				\begin{tabular}{c}
					$C_{\gamma}H(u,v)\geq u^2$\\
					$u^2+v^2\geq D_{\gamma}$
				\end{tabular}\right.
			\]
			$\gamma \mapsto C_\gamma$ is nonincreasing and $\lim_{\gamma \rightarrow 0} C_\gamma = \infty$,
	\item  \label{lem:point6}there are $\overline{\theta}\in(0,\pi/4)$ and $\overline{v}>0$ such that  for every $(u,v)$ which satisfy $H(u,v)\geq 0$ and $|u|\leq \tan(\overline{\theta})|v|$, then we have $|v|\geq \overline{v}.$ For any $\overline{\theta}\in(0,\pi/4)$, there are $P\in(0,1/2)$ and $\overline{v}>0$ such that this point remains true for all $p\in(0,P)$. 
\end{enumerate}

\end{lem}
We define $E_0=(\Omega-m)^\frac{1}{2(1-p)}= \sup\{v| \exists u:~ H(u,v)=0\}$.
\begin{figure}
\includegraphics[scale=0.9]{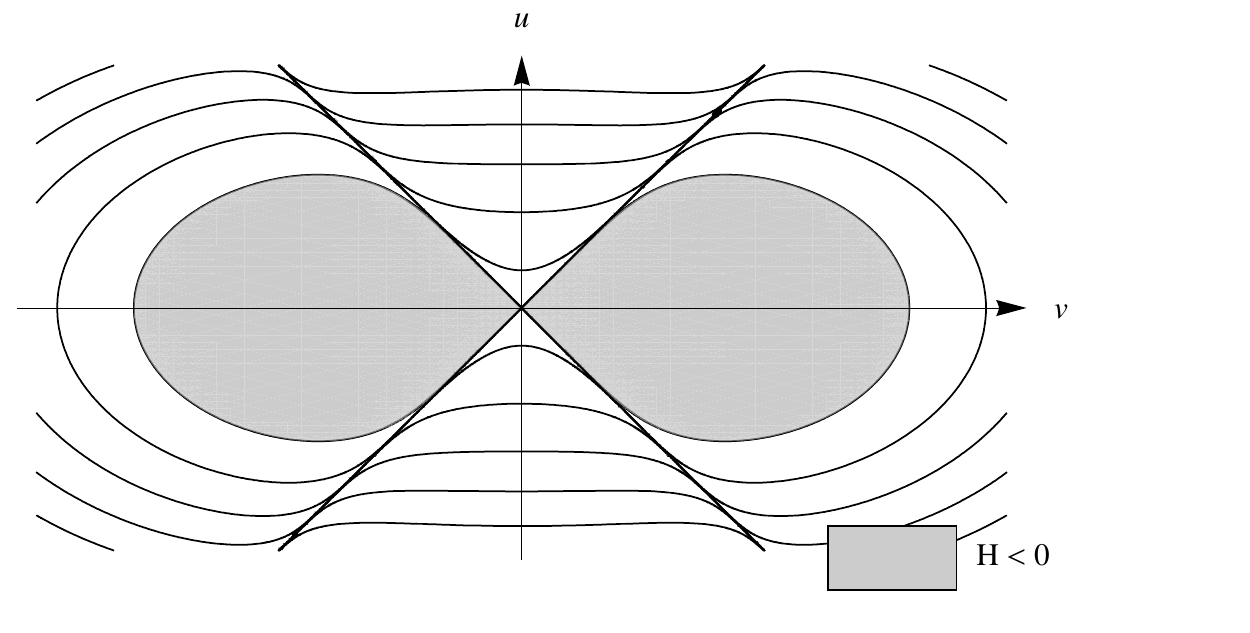}
\caption{Energy levels}
\label{graph:level}
\end{figure}
The proof is a straightforward calculation and is postponed in the appendix. 
%
%

\subsection{The regularized problem}\label{subsection:regul}
$F$ is so singular at zero that the main O.D.E. theorems \cite{CL55} fail to show existence and local uniqueness for problem \eqref{systemequation:diag}. To overcome this, we  introduce a regularized problem.

Let $E_1$ be a positive constant that will be fixed later. We define for $\epsilon \in (0,m)$ the sets (figure \ref{graph:regul}):
\begin{equation*}
\begin{aligned}
	\mathcal{R}^1_{\epsilon}=&\left\{(u,v)\in \mathbb{R}^2:~|v-u| \leq E_1, H_{\epsilon}(u,v)\geq 0,H_{\epsilon}(v,u)\geq 0\right\}\\
	\mathcal{R}^2_{\epsilon}=&\left\{(u,v)\in \mathbb{R}^2:~|v-u| \leq E_1/2, H_{\frac{\epsilon}{2}}(u,v)\geq 0,H_{\frac{\epsilon}{2}}(v,u)\geq 0\right\}.\\
\end{aligned}
\end{equation*}
Let us remark that by lemma \ref{energielevel}, we have
\[
	\mathcal{R}^2_{\epsilon}\subset\mathcal{R}^1_{\epsilon}\subset H^{-1}((0,+\infty))\cup \{0\}
\]
and
\[
	\cap_{\epsilon >0}\mathcal{R}^1_{\epsilon} = \{|u|=|v|\}.
\]
%
\begin{figure}
\includegraphics[scale=0.8]{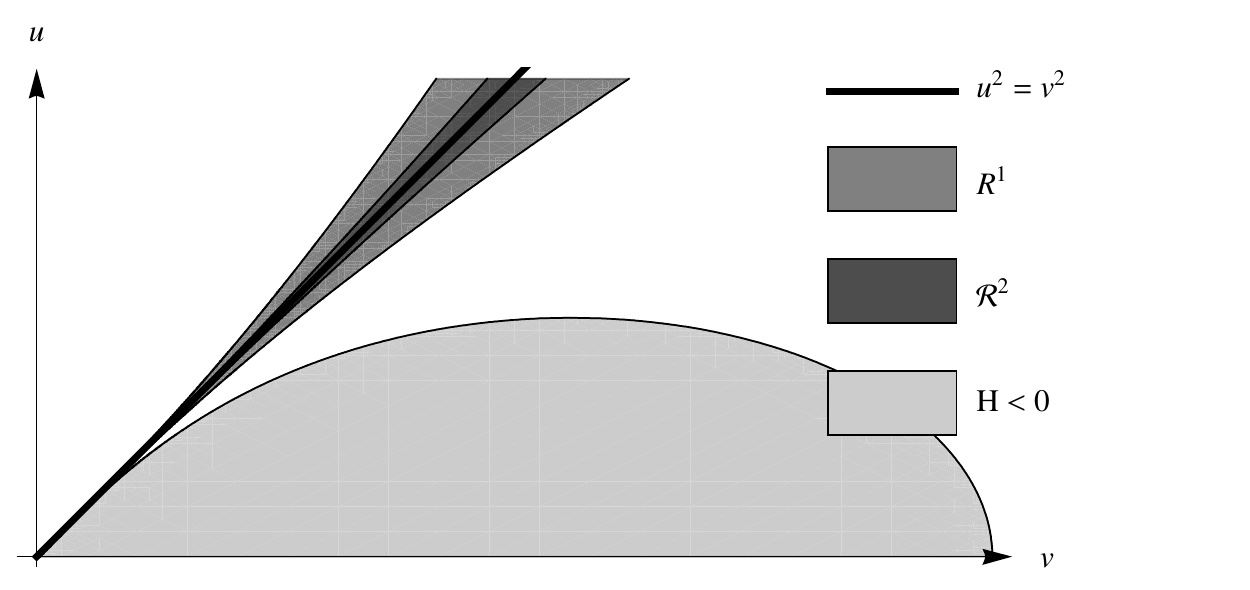}
\caption{Regularization area}
\label{graph:regul}
\end{figure}
Let $\varphi_{\epsilon}:\mathbb{R}^2\backslash\{(0,0)\}\rightarrow [0,1]$ be a smooth function such that 
\[\varphi_{\epsilon} \equiv \left\{
\begin{array}{l}
	0 \mbox{~on~}  \mathcal{R}^2_{\epsilon}\\
	1 \mbox{~on~}  (\mathcal{R}^1_{\epsilon})^c\\

\end{array} \right.
\]
and
\[
	\varphi_\epsilon(u,v) = \varphi_\epsilon(|u|,|v|)=\varphi_\epsilon(|v|,|u|)
\]
for any $(u,v)$ in $\mathbb{R}^2\backslash\{(0,0)\}.$
We will then study the following system of equations:
\begin{equation}  \label{equation:uvreg}
	\left\{ \begin{array}{lcr}
  u' + \displaystyle\frac{2u}{r}\varphi{}_\epsilon(u,v) &= & v(-F(v^2-u^2)-(m-\Omega)) \\
  v'  & = & u(-F(v^2-u^2)-(m+\Omega)).
\end{array} \right.
\end{equation}
Let us remark that there is no regular solution of the Cauchy problem:
\begin{equation}\label{systequationdiag}
	\left\{\begin{array}{ll}
		\eqref{equation:uvreg}\\
		u(R)=v(R)=x
	\end{array}\right.
\end{equation}
where $R>0$ and $x\in\mathbb{R}^*.$  Nevertheless, near the set $\{|u|=|v|\}$ the system of equations is autonomous and hamiltonian. This will allow us to get existence and local uniqueness while keeping the qualitative properties of the solutions to problem \eqref{systequation:uv} that we need for the shooting method.


%
%
%
\subsection{Qualitative results}
We assume in this part, that $(u,v)$ is a solution in the extended sense of equations \eqref{equation:uvreg} defined on an interval $I$ and $\epsilon>0.$
As in \cite{balabane1990}, we show, in the next lemma, that the function $H(u,v)$ is nonincreasing.
\begin{lem}
\label{deriveeenergie}
We have for $r\in I$:
\[\frac{d}{dr}H(u,v)(r)  = -\frac{2u^2\varphi{}(u,v)}{r}(p| v^2-u^2|^{p-1}+(m+\Omega)),\]
so $r\mapsto H(u,v)$ is nonincreasing. 
\end{lem}
\begin{proof}
We have for $r\in I$:
\begin{align*}
\frac{d}{dr}H(u,v)(r)	&= -p(v'v-u'u)| v^2-u^2|^{p-1} + v'v(\Omega-m)+ u'u(\Omega+m)\\
					&= v'(u' + \displaystyle\frac{2u}{r}\varphi{}(u,v)) -u'v'\\
					&= v'\left(\displaystyle\frac{2u}{r}\varphi{}(u,v)\right)\\
					&=-\frac{2u^2\varphi{}(u,v)}{r}(p| v^2-u^2|^{p-1}+(m+\Omega)).
\end{align*}
\end{proof}
%
In the next lemma, we study the speed of rotation of the trajectory of $(u,v)$ around zero.
\begin{lem}
\label{deriveeangulaire}
Let $\theta(r) = \int_0^r\frac{d}{dr}\arctan(\frac{u}{v})(s)ds$. We have for $r\in I$:
\[
	\theta(r)' =  \frac{pH(u,v) + (1-p)(v^2(\Omega-m))+u^2(\Omega+m))- \displaystyle 2uv\varphi(u,v)/r}{u^2+v^2} .
\]
If $H(u,v)(r)\geq 0$, we get moreover 
\[
	\theta(r)' \geq (1-p)(\Omega-m)-\varphi(u,v)/r.
\]
\end{lem}
\begin{proof}
We have for $r\in I$:
\begin{align*}
&\theta(r)'= \frac{u'v-v'u}{u^2+v^2}\\
		&= \frac{ -p(v^2-u^2)| v^2-u^2|^{p-1} + v^2(\Omega-m))+u^2(\Omega+m)-  2uv\varphi(u,v)/r }{u^2+v^2}\\
		&= \frac{pH(u,v) + (1-p)(v^2(\Omega-m))+u^2(\Omega+m))-  2uv\varphi(u,v)/r}{u^2+v^2}.
\end{align*}
If we assume moreover that $H(u,v)(r)\geq 0$, we get
\begin{align*}
		\lefteqn{\theta(r)'\geq \frac{(1-p)(\Omega-m)(v^2+u^2)- 2uv\varphi(u,v)/r}{u^2+v^2}}\\
		&\geq (1-p)(\Omega-m)-\varphi(u,v)/r.
\end{align*}
\end{proof}

\subsection{Existence and local uniqueness results}
Let  $\epsilon \in (0,m)$. 
%
%
We show, in the next lemma, that problem \eqref{systequationdiag} has a unique local solution.
\begin{lem}
\label{existanceunicite}
Consider the initial value problem for $r_1\geq0,$ $(u_1,v_1)\in\mathbb{R}^2\backslash \{0\}$:
\begin{equation}\label{systequationlocal:uvreglem}
\left\{
\begin{array}{l}
 	(\ref{equation:uvreg}) \\
 	(u,v)(r_1)  = (u_1,v_1),
\end{array}
\right.
\end{equation} 
such that $ u_1 = 0$ if $r_1=0$.
Then, there exists a unique local solution.
\end{lem}
\begin{proof}
The Cauchy-Lipschitz theorem shows this result provided that $r_1 \neq 0$ and $|u_1|\neq |v_1|.$ The contraction mapping argument sketched in \cite{balabane1990} ensures the result for $r_1=0.$ So the point is when $|u_1|= |v_1|.$ In this case, the system (\ref{equation:uvreg}) reduces itself into the autonomous hamiltonian system of equations: 
\begin{equation}  \label{equation:uvautonome} 
\left\{ \begin{array}{ll}
  u'  =& v(-p| v^2-u^2|^{p-1}-(m-\Omega)) \\
  v'  =& u(-p| v^2-u^2|^{p-1}-(m+\Omega)).
\end{array} \right.
\end{equation}
We assume:
\[
	u_1=v_1>0.
\]
We choose 
\[
	(u_2,v_2)\in H^{-1}(H(u_1,v_1))\cap \{0<u<v\}\cap \mathcal{R}^2_\epsilon.
\]
Then, the solution $(u,v)$ to problem
\[
	\left\{
		\begin{array}{l}
			\eqref{equation:uvautonome}\\
			(u(0),v(0)) = (u_2,v_2)
		\end{array}
	\right.
\]
is well-defined, $\mathcal{C}^1$ on a maximal interval $[0,r_2)$, locally unique and 
\[
	H(u(r),v(r))=H(u_1,v_1)
\]
 for each $r\in[0,r_2)$ by lemma \ref{deriveeenergie}. Since the set
\[
	\{(u,v),~H(u,v) = H(u_1,v_1) \}
\]
is compact and 
\[
	\theta(r)'= \frac{pH(u,v) + (1-p)(v^2(\Omega-m))+u^2(\Omega+m))}{u^2+v^2}\geq (1-p)(\Omega-m)>0
\]
 for $r\in[0,r_2)$ by lemma \ref{deriveeangulaire}, we get that
\[
	\underset{r\rightarrow r_2}{\lim}~(u,v)(r)=( u_1,v_1).
\]
We define then $(\tilde{u},\tilde{v})$ by
\[
	(\tilde{u},\tilde{v})(r) = (u,v)(r-r_1+r_2)
\]
for any $r\in[r_1-r_2,r_1],$ it solves problem \eqref{systequationlocal:uvreglem} on the interval $[r_1-r_2,r_1].$ We proved the existence and the local uniqueness on one side. The same argument works as well for the remaining cases.
\end{proof}
%
%
%
We will now show the existence and the uniqueness of the maximal solution in the extended sense of the regularized problem.
\begin{lem}
\label{exitenceglobale}
For each $x>0$, there is a unique solution $(u,v)$ of the problem 
\begin{equation}\label{systequation:uvreg}
	\left\{\begin{array}{l}
		\eqref{equation:uvreg}\\
		(u(0),v(0))=(0,x)
	\end{array}\right.
\end{equation}
on an interval $[0,R_x)$ with $R_x\in(0,+\infty]$ such that 
\[
	(u(r),v(r))\ne (0,0),~\forall~r\in[0,R_x),
\]
and $R_x$ is maximal for this property.
 There is a positive constant $C$ such that $(u,v)$ satisfies
\[  
	\underset{r\in[0,R_x)}{\sup}\sqrt{u^2(r)+v^2(r)}\leq C.
\]
	 Moreover, if $R_x<\infty$, then we have
\[
	\underset{r\rightarrow R_x}{\lim}~(u(r),v(r)) = (0,0).
\]
\end{lem}
\begin{proof}
Lemma \ref{existanceunicite} ensures the existence and the uniqueness of the maximal solution $(u,v)$ of problem \eqref{systequation:uvreg}.
 The function $H(u,v)$ is non increasing in $[0,R_x)$ by lemma \ref{deriveeenergie}, so there exists a positive constant $C$ which depends only on $H(0,x),$ such that 
\[  
	\underset{r\in[0,R_x)}{\sup}\sqrt{u^2(r)+v^2(r)}\leq C
\]
by lemma \ref{energielevel}. Any solution to problem \eqref{systequation:uvreg} can then be extended thanks to lemma \ref{existanceunicite} as far as $(u,v)(r)\neq 0.$  Thus, if $R_x<\infty$, we have
\[
	\underset{r\rightarrow R_x}{\lim}~(u(r),v(r)) = (0,0).
\]
\end{proof}
%
\section{The shooting method}\label{section:shoot}
In this section, we fix $p\in(0,1/2).$
Once the existence and the local uniqueness are shown for the regularized problem, we can adapt the shooting method of Balabane, Dolbeault and Ounaies \cite{balabane2003} to our problem. In this section, we will denote by $(u_x,v_x)$ the maximal solution of problem \eqref{systequation:uvreg} where $x>0$, to insist on the dependence on $x$. 

We define 
\[
	N_x(a,b) := \#\{r\in(a,b)|v_x(r)=0 \}\in[0,+\infty]
\] 
for $0\leq a<b\leq+\infty $ and for every $\gamma \geq 0$, 
\[
	\rho_x(\gamma) := \sup\{r>0: H(u_x, v_x)\geq \gamma\}\in[0,+\infty]\cup\{-\infty\}.
\]  
We will write $N_x(b)$ instead of $N_x(0,b).$
The core of the shooting method will be the study of the following sets which are introduced in \cite{balabane2003}:
\begin{defn}\label{setsinitialdata} 
Let $k\in\mathbb{N}.$ We define
\begin{align*}
\lefteqn{A_k :=\{ x>0|\underset{r\rightarrow\infty}{\lim} H(u_x,v_x)(r) <0,(u_x,v_x)(r)\neq (0,0)~\forall r \geq0, \dots}\\
	&&N_x(\infty) = k \}\\
&I_k  :=\{ x>0|\underset{r\rightarrow\rho_x(0)}{\lim} (u_x,v_x)(r) =(0,0), ~N_x(\rho_x(0)) = k\}. 
\end{align*}
\end{defn}

\begin{figure}[!ht]
\begin{center}
\includegraphics[scale=0.9]{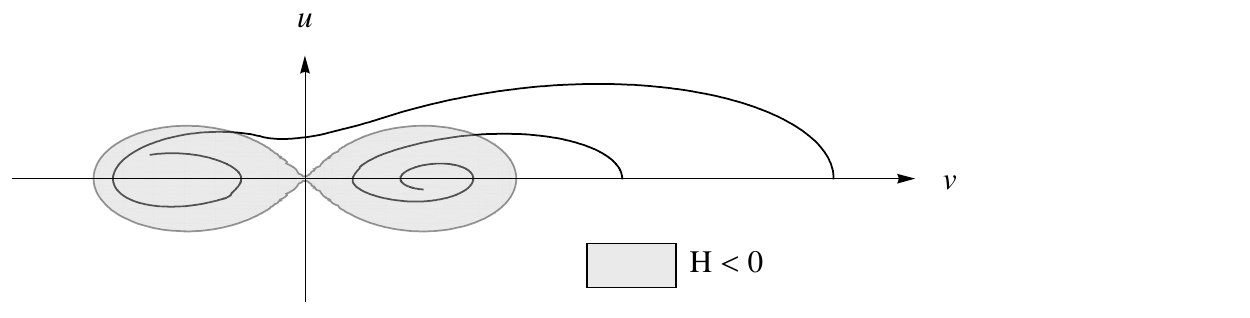}
\caption{Solutions belonging to $A_0$ and $A_1.$}
\label{graph:shooting}
\end{center}
\end{figure}
\begin{rem}\label{rem:shootingsets}
If there is $R>0$ such that $H(u_x,v_x)(R)<0$ then by lemmas \ref{deriveeenergie} and \ref{exitenceglobale}, we get that $R_x=+\infty.$ So, the sets $A_k$ are well defined for any $k$. The key idea is that $(u_x,v_x)$ winds around the connected set $\{H=0\}$ and cross it at finite radius $\rho_x(0)$. We can have
\[
	\underset{r\rightarrow R_x}{\lim} H(u_x,v_x)(r)<0
\] 
and $x$ belongs to $\underset{k\in\mathbb{N}}{\cup}~A_k$ or $R_x<\infty,$ 
\[
	\underset{r\rightarrow R_x}{\lim} H(u_x,v_x)(r)=0
\] 
and $x$ belongs to $\underset{k\in\mathbb{N}}{\cup}~I_k.$ The goal is to show that $I_k$ is not an empty set for any $k$.
\end{rem}
\subsection{Main results}
%
%
%
%
In the next lemma, we give uniform estimations far enough from $r=0$ and from the radius 
\[
	\inf\{r>0,~|u_x|(r)=|v_x|(r)\}.
\]
\begin{lem}
\label{estimationangulaire}
There exist $\overline{p}\in (0,1/2)$ and for all $0<q<\overline{p},$ a constant $\Omega_q>m$ such that if $\Omega > \Omega_q,$  then there are $r_0>0,~ \alpha >0$ and $\Theta >0$ which satisfy:
\begin{enumerate}[(i)]
	\item $r_0 > \frac{1}{(\Omega-m)(1-p)},$
	\item $v_{}^2(r)-u_{}^2(r) \geq \alpha x^2$ for all $r\in[0,r_0],$
	\item $\theta_{}(r)'\geq \Theta,$	whenever  $r\geq r_0$ and $H(u_{},v_{})(r)\geq0,$
\end{enumerate}
for any  $0<p\leq q$ and $0<x$ where  $(u_{},v_{})$ is the solution of problem \eqref{systequation:uv} with initial condition $(0,x)$ ($\overline{p}\simeq 0.0173622$).
\end{lem}
\begin{rem}
$r_0,$ $\Theta,$ $\alpha$ and $\Omega_q $ may be chosen independent on $p$ and $x$.
\end{rem}

This result is slightly finer than a result of Balabane, Cazenave and Vazquez \cite{balabane1990} but the proof is straightforward and based on their ideas. It is postponed in the appendix.

From now on, we fix $p<\overline{p},$ $\Omega_p<\Omega$ and  $E_1=E_0^2\alpha$ where $E_1$ is the constant in the regularization sets of subsection \ref{subsection:regul}, $E_0$ comes from lemma \ref{energielevel} and $\alpha$ from lemma \ref{estimationangulaire}. 
\begin{rem}
Lemma \ref{estimationangulaire} remains true for the solutions of problem \eqref{systequation:uvreg} thanks to these choices for $E_1$ and for the regularization.
\end{rem}
We study now the winding number $N_x(a,b).$
\begin{lem}\label{controlenombredetour}
Let us assume that $r_0< \rho_x(0)$ and let $r_0\leq a < b\leq \rho_x(0)$, then:
\[N_x(a,b)\geq \lfloor\frac{\Theta}{\pi}(b-a)\rfloor\]
where $\lfloor.\rfloor$ is the floor function.
\end{lem}
\begin{proof}
By lemma \ref{estimationangulaire}, we get  
\[
	\theta_x(b)-\theta_x(a) = \int^b_a \theta_x(s)'ds\geq \Theta(b-a),
\]
so that $N_x(a,b)\geq \lfloor\frac{\Theta}{\pi}(b-a)\rfloor.$
\end{proof}

Let $\gamma >0$ be such that $q:=2(1-p)(\Omega+m)C_\gamma-1>0.$ Such a $\gamma$ exists by lemma \ref{energielevel}.
The following lemma gives estimations on the decay of $H(u_x,v_x).$
\begin{lem}\label{controledecroissanceenergie}
Let us assume that $\rho_x(\gamma)>0$, then we have for $r\in (0,\rho_x(\gamma))$:
\[\frac{d}{dr}\left(r^{2(\Omega+m)C_\gamma}H(u_x,v_x)\right)^{1-p}\geq -2p(1-p)C_\gamma^p\rho(\gamma)^{q}|1-2\sin^2\theta_x|^{p-1}. \]
\end{lem}
\begin{proof}
For the sake of notation simplicity, we remove here the $x$ subscripts. By lemma \ref{deriveeenergie}, we have for $r\in(0, \rho(\gamma))$:
\begin{align*}
\frac{d}{dr}H(u,v) 		&\geq-\frac{2u^2}{r}(p| v^2-u^2|^{p-1}+(m+\Omega)) \\
					&\geq -\frac{2C_\gamma(\Omega+m) H(u,v)}{r} - \frac{2pu^2}{r(u^2+v^2)^{1-p}}\left|\frac{v^2-u^2}{u^2+v^2}\right|^{p-1}\\
					&\geq-\frac{2C_\gamma(\Omega+m) H(u,v)}{r} - \frac{2pu^{2p}}{r}\left(\frac{u^{2}}{u^2+v^2}\right)^{1-p}| 1-2\sin^2\theta|^{p-1}\\
					&\geq-\frac{2C_\gamma(\Omega+m) H(u,v)}{r} - \frac{2pC_\gamma^pH(u,v)^p| 1-2\sin^2\theta|^{p-1}}{r},
\end{align*}
so,
\begin{align*}
\frac{d}{dr}(r^{2(\Omega+m)C_\gamma}&H(u,v))\\
 	&\geq -2pC_\gamma^pH(u,v)^p| 1-2\sin^2\theta|^{p-1}r^{2(\Omega+m)C_\gamma-1}\\
												&\geq -2pC_\gamma^p| 1-2\sin^2\theta|^{p-1}r^{q}(r^{2(\Omega+m)C_\gamma}H(u,v))^p \\
												&\geq -2pC_\gamma^p| 1-2\sin^2\theta|^{p-1}\rho(\gamma)^q(r^{2(\Omega+m)C_\gamma}H(u,v))^p. \\				
\end{align*}
Finally, we get
\begin{align*}
	\frac{d}{dr}\left(r^{2(\Omega+m)C_\gamma}H(u_x,v_x)\right)^{1-p}&=(1-p)\frac{\frac{d}{dr}(r^{2(\Omega+m)C_\gamma}H(u,v))}{r^{2(\Omega+m)C_\gamma}H(u,v))^p}\\
	&\geq -2p(1-p)C_\gamma^p| 1-2\sin^2\theta|^{p-1}\rho(\gamma)^q.
\end{align*}
\end{proof}
%
%
%
The following proposition ensures that the number of times the solutions circle around the set $\{H=0\}$ tends to infinity when $x$ tends to infinity.
\begin{prop}
\label{proprietedenroulement}
We have $\lim_{x\rightarrow\infty} N_x(r_0,\rho_x(\gamma)) = \infty$.
\end{prop}
\begin{proof}
Lemma \ref{estimationangulaire} ensures that
\[
	v_x^2(r_0)-u_x^2(r_0)\geq\alpha x^2
\]
and lemma \ref{energielevel} gives
\begin{align*}
	\underset{x\rightarrow +\infty}{\lim}~H(u_x,v_x)(r_0)&\geq \underset{x\rightarrow +\infty}{\lim}~A(u_x^2(r_0)+v_x^2(r_0))-B \\
	&\geq \underset{x\rightarrow +\infty}{\lim}~A\alpha x^2-B= +\infty.
\end{align*}
Then, there exists $x_0>0$ such that 
\[
	\rho_x(0)>\rho_x(\gamma)>r_0
\]
for any $x\geq x_0.$
By lemma \ref{controledecroissanceenergie}, we get
\begin{align*}
	\lefteqn{\left(\rho_x(\gamma)^{2(\Omega+m)C_\gamma}\gamma\right)^{1-p} -\left(r_0^{2(\Omega+m)C_\gamma}H(u_x,v_x)(r_0)\right)^{1-p}}\\
	&= \int_{r_0}^{\rho_x(\gamma)} \frac{d}{dr}\left(r^{2(\Omega+m)C_\gamma}H(u_x,v_x)\right)^{1-p}dr\\
	&\geq \int_{r_0}^{\rho_x(\gamma)} -2p(1-p)C_\gamma^p\rho_x(\gamma)^{q}|1-2\sin^2\theta_x(r)|^{p-1} dr.
\end{align*}
Lemma \ref{estimationangulaire} gives then that:
\begin{align*}
	\lefteqn{\left(\rho_x(\gamma)^{2(\Omega+m)C_\gamma}\gamma\right)^{1-p} -\left(r_0^{2(\Omega+m)C_\gamma}H(u_x,v_x)(r_0)\right)^{1-p}}\\
	&\geq -\frac{2p(1-p)C_\gamma^p\rho_x(\gamma)^{q} }{\Theta}\int_{r_0}^{\rho_x(\gamma)}| 1-2\sin^2\theta_x(r)|^{p-1}\theta'_x(r)dr\\
	&\geq -\frac{2p(1-p)C_\gamma^p\rho_x(\gamma)^{q} }{\Theta}\int_{\theta_x(r_0)}^{\theta_x(\rho_x(\gamma))}| 1-2\sin^2\theta|^{p-1}d\theta\\
	&\geq -\frac{2p(1-p)C_\gamma^p\rho_x(\gamma)^{q} }{\Theta}(N_x(r_0,\rho_x(\gamma))+2)\int_{0}^{\pi}| 1-2\sin^2\theta|^{p-1}d\theta.\\
\end{align*}
Since $p\in(0,1)$, the integral $\int_{0}^{\pi}| 1-2\sin^2\theta|^{p-1}d\theta$ converges. Moreover, we have by lemma \ref{controlenombredetour} that
\[
	\frac{\pi(N_x(r_0,\rho_x(\gamma))+1)}{\Theta}+r_0\geq \rho_x(\gamma).
\]
We have already shown that
\[
	\underset{x\rightarrow +\infty}{\lim}~H(u_x,v_x)(r_0)=+\infty
\]
so these inequalities ensure that
\[
	\underset{x\rightarrow +\infty}{\lim}~N_x(r_0,\rho_x(\gamma))=+\infty.
\] 
\end{proof}
We have now to construct the trapping zone as in proposition $3$ of \cite{balabane2003}. Nevertheless, the zone we construct is more complicated (see figure \ref{graph:trapping}).
\begin{prop}
\label{trapping}
For all $k \in \mathbb{N},$ there exists $\sigma >0$ such that if $N_x(\overline{R}) = k$ and $u^2_x(\overline{R})+v^2_x(\overline{R})< \sigma^2$ for some $x$ and $\overline{R}$ positive, then $x$ belongs to $A_k\cup I_k\cup A_{k+1}.$ 
\end{prop}
%
%
%
\begin{proof}
For the sake of clarity, we remove here the $x$ subscripts. The decay of the energy in lemma \ref{deriveeenergie} makes the result obvious if $H(u,v)(\overline{R}) \leq 0$. By symmetry, we can assume without loss of generality that
\[
	(u,v)(\overline{R}) \in \left\{u>0\right\}\cup\left\{H>0\right\}~\mbox{and}~x>E_0
\]
where $E_0$ is defined in lemma \ref{energielevel}.
Let $(M_1,M_2)$ be the unique point of 
\[
	H^{-1}(\{0\})\cap\{v>u>0\}\cap\{v^2-u^2=(\frac{\Omega-m}{p})^{\frac{1}{p-1}}\}.
\]
It exists since $(\frac{\Omega-m}{p})^{\frac{1}{p-1}}<E_0^2$ (see figure \ref{graph:trapping}). Let $v_1>0$ such that $(M_1/2,v_1)$ is the unique point of $\partial\mathcal{R}^1_{\epsilon}\cap\{v>u>0\}.$ We define
\[
	K := H(M_1/2,v_1)>0.
\]
The parameter $\sigma$ is chosen such that:
\begin{equation}\label{lemeq:sigma}
	0<\sigma  <  \min\{\frac{M_1^3}{4M_2(\frac{\pi(k+2)}{\Theta}+r_0)}, \sqrt{\alpha E_0^2}\}
\end{equation}
and $B(0,\sigma) \subset H^{-1}(-\infty, K)\cap\{u<M_1/2\}$ where $B(0,\sigma)$ is the euclidean ball of $\mathbb{R}^2$ centered in $0$ and of radius $\sigma.$ Let 
\[\mathcal{D} = \left\{(u,v):~0<H(u,v)<K,~0<u<M_1\right\},\]
(see figure \ref{graph:trapping}). We have that
\[
	(u,v)(\overline{R})\in \mathcal{D}.
\]
If $(u,v)$ exits $\mathcal{D}$ crossing the boundary at $\{H=0\}$, we have the result. We now prove that this is the only possible way to exit $\mathcal{D}$. 
\begin{figure}[!ht]
\begin{center}
\includegraphics[scale=0.9]{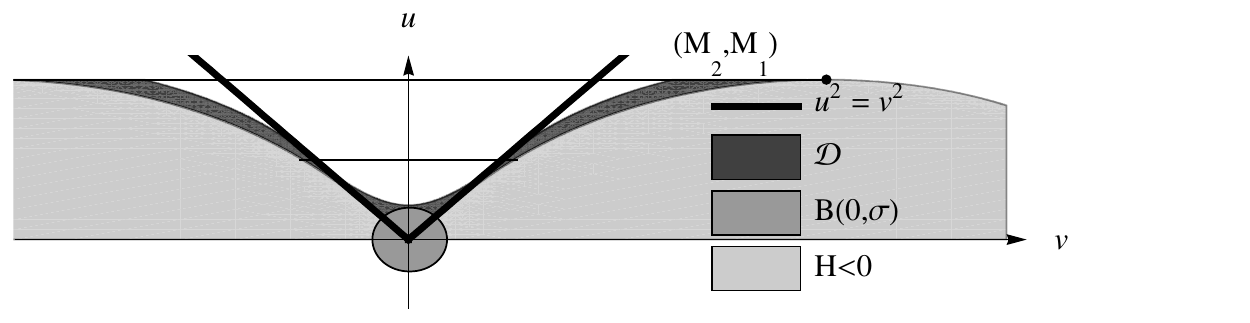}
\caption{Trapping region}
\label{graph:trapping}
\end{center}
\end{figure}
Let us assume by contradiction that $(u,v)$ do not cross the boundary of $\mathcal{D}$ at $\{H=0\}$, then by lemmas \ref{deriveeenergie} and \ref{deriveeangulaire}, $(u,v)$ must exit $\mathcal{D}$ at  $\{u=M_1\}$. 
We define $R''$ by
\[
	R'':= \inf\{r>\overline{R}:~(u,v)(r)\notin \mathcal{D}\}\in(\overline{R},+\infty].
\]
We have $H(u,v)(R'')>0$ by assumption. 
Since $x>E_0$, lemma \ref{estimationangulaire} ensures that 
\[
	v^2(r)+u^2(r)\geq v^2(r)-u^2(r)\geq \alpha x^2\geq \alpha E_0^2
\]
for any $r\in[0,r_0].$ By inequality \eqref{lemeq:sigma}, we get that $\overline{R}>r_0$ and the third point of lemma \ref{estimationangulaire} ensures that $R''<+\infty.$ We have moreover 
\[
	k = N(\overline{R})\leq N(R'')\leq k+1,
\]
lemma \ref{controlenombredetour} ensures that
\[
	k+1\geq N(R'')\geq N(r_0,R'')\geq \lfloor\frac{\Theta}{\pi}(R''-r_0)\rfloor,
\]
so
\begin{equation}\label{eq:estimationradiustrapping}
	R''\leq \frac{\pi(k+2)}{\Theta}+r_0.
\end{equation}
We define now
\[
	R':= \sup\{r\in(\overline{R},R'') |~|u|<M_1/2\}.
\]
It is well-defined and $R'\in(\overline{R},R'')$ because 
\[
	|u(\overline{R})|\leq \sqrt{u^2(\overline{R})+v^2(\overline{R})}< \sigma<M_1/2~\text{and}~u(R'')=M_1.
\]
By lemma \ref{deriveeenergie}, we have that
\[
	0<H(u,v)(r)\leq H(u,v)(\overline{R})<K
\]
and by lemma \ref{estimationangulaire} that
\[
	u(r)\geq 0\geq v(r)
\]
for any $r\in[R',R''].$ By the construction of $\mathcal{D},$  we have that
\[
	\{(u,v):~M_1/2<|u|<M_1,~0<H(u,v)<K\}\cap \mathcal{R}^1_\epsilon=\emptyset
\]
thanks to our choices of $v_1,$ $K$ and the symmetries of $\varphi_\epsilon.$ Thus, we get
\[
	\varphi_\epsilon(u,v)(r) = 1
\]
for any $r\in[R',R''].$ By lemma \ref{deriveeenergie}, inequality \eqref{eq:estimationradiustrapping} and the definitions of $R'$ and $R'',$ we obtain
\begin{align*}
	\lefteqn{H(u,v)(R')-H(u,v)(R'') =-\int_{R'}^{R''}\frac{d}{dr}\left(H(u,v)(r)\right)dr}\\
	&=\int_{R'}^{R''} \frac{2u^2}{r}(p| u^2-v^2|^{p-1}+(m+\Omega))dr\\
	&\geq \frac{M_1^2}{2(\frac{\pi(k+2)}{\Theta}+r_0)}\int_{R'}^{R''} (p| u^2-v^2|^{p-1}+(m+\Omega))dr.
\end{align*}
Then, we have
\begin{align*}
	\lefteqn{M_1/2 = u(R'')-u(R')}\\
	&= \int_{R'}^{R''}\left[ -v(p| u^2-v^2|^{p-1}+(m+\Omega))- \frac{2u}{r}\right]dr\\
	&\leq M_2 \int_{R'}^{R''} (p| u^2-v^2|^{p-1}+(m+\Omega))dr
\end{align*}
and
\begin{align*}
	\lefteqn{H(u,v)(R'')\leq H(u,v)(R')- \frac{M_1^3}{4M_2(\frac{\pi(k+2)}{\Theta}+r_0)}}\\
	&\leq H(u,v)(\overline{R})-\sigma\leq 0.
\end{align*}
This contradicts $H(u,v)(R'')>0$. 
\end{proof}
We show in the next lemma that the solution can be controlled uniformly in some Sobolev norm.
\begin{lem}\label{sobolev}
Let us assume that there are $\overline{R}>r_0$ and $y>0$ such that
\[
	H(u_x,v_x)(r)>0~\text{and}~\sqrt{u^2_x(r)+v^2_x(r)} \leq y
\] 
for any $r\in[r_0,\overline{R}]$ where $(u_x,v_x)$ is the solution of problem \eqref{systequation:uvreg} with $x>0$. 
Then, for all $s \in (1,\frac{1}{1-p})$, there exists $C>0$ such that for all $r_0<r_1<r_2<\overline{R}:$
\begin{align*}
	\lefteqn{\|u_x^2\|_{W^{1,s}(r_1,r_2)},\|v_x^2\|_{W^{1,s}(r_1,r_2)},\|u_xv_x\|_{W^{1,s}(r_1,r_2)}}\\
	&&\leq C(N_x(r_1,r_2)+ (r_2-r_1)+ 1),
\end{align*}
$C$ does not depend on $\epsilon$ and $x$. 
\end{lem}
\begin{proof}For the sake of clarity, we remove here the $x$ subscripts. We have:
\begin{align*}
	\frac{du^2}{dr}	&= 2uv(-p| u^2-v^2|^{p-1} +(\Omega-m)) - \frac{2u^2\varphi(u,v)}{r},
\end{align*}
so, for $r_0<r_1<r_2<\overline{R},$
\begin{align*}
	\int_{r_1}^{r_2}\left|\frac{du^2}{dr}\right|^sdr	&\leq C\left( \int_{r_1}^{r_2}\left|\frac{uv}{(u^2+v^2)^{1-p}}\right|^s\left|\frac{u^2-v^2}{u^2+v^2}\right|^{s(p-1)}dr+(r_2-r_1)\right)\\
										&\leq C\left( \int_{r_1}^{r_2}| 1-2\sin^2\theta(r)|^{s(p-1)}dr + (r_2-r_1)\right)\\
										&\leq C\left( \int_{r_1}^{r_2}\frac{1}{\Theta}| 1-2\sin^2\theta(r)|^{s(p-1)}\theta'(r)dr + (r_2-r_1)\right)
\end{align*} 
by lemma \ref{estimationangulaire}. Since $s(p-1)>-1,$ the integral
\[
	\int_0^\pi | 1-2\sin^2\theta|^{s(p-1)}d\theta 
\]
converges and
\[
	\int_{r_1}^{r_2}\left|\frac{du^2}{dr}\right|^sdr\leq C(N(r_1,r_2)+ (r_2-r_1)+ 2).
\]
The same proof works as well for $v^2$ and  $uv$.
\end{proof}
%
We study in the following lemma, the dependance of the solutions on the initial conditions. This is a very important point in the shooting method and that is the reason why we introduced a regularized problem.
\begin{lem}\label{continuity}
Let $(\overline{u},\overline{v})$ be a solution of \eqref{equation:uvreg} defined on an interval $[r_1,r_2]$ such that 
\[
	(\overline{u},\overline{v})(r)\neq (0,0)
\] 
for all $r\in[r_1,r_2].$ For all $\eta >0$, there exists $\delta >0,$ such that if $(u,v)$ is also a solution of \eqref{equation:uvreg} with 
\[
	\|(u,v)(r_1)-(\overline{u},\overline{v})(r_1)\|< \delta,
\] 
we have 
\[
	\|(u,v)(r)-(\overline{u},\overline{v})(r)\|< \eta
\] 
for all $r\in[r_1,r_2]$ where $\|.\|$ is the euclidian norm of $\mathbb{R}^2$.
\end{lem}
The proof is the same as in \cite{balabane2003} and follows from lemmas \ref{existanceunicite} and \ref{sobolev}.

%
%
We will now study the behaviors at infinity.
\begin{lem}
\label{comportementinfini}
For all $x>0$, we have:
\begin{enumerate}[(i)]
	\item either $N_x(R_x) < \infty,$
	\item or $R_x=+\infty,$ $N_x(\infty) = \infty$, $\underset{r\rightarrow \infty}\lim H(u_x,v_x)=0$ and $(u_x,v_x)\neq 0$ on $[0,\infty)$.
\end{enumerate}
\end{lem}
\begin{proof}
For the sake of notation simplicity, we remove here the $x$ subscripts. We recall that $R$ is the largest radius before the solution hits $0$. We study now the different cases.
\begin{enumerate}[(a)]
	\item We assume that $R<\infty$. By lemma \ref{exitenceglobale}, we have 
	\[
		H(u,v)(r)\geq0
	\]
	for all $r\in(0,R).$ We have by lemmas \ref{energielevel} and \ref{deriveeangulaire} that for any $R\geq r\geq r_0$ such that $|u|(r)\leq \tan(\overline{\theta})|v|(r):$
	\begin{align*}
		\theta(r)' 	&=	  \frac{pH(u,v) + (1-p)(v^2(\Omega-m))+u^2(\Omega+m))- 2uv\varphi(u,v)/r}{u^2+v^2}\\
				&\leq \frac{pH(u,v)}{\overline{v}^2}+(1-p)(\Omega+m)+\frac{2|uv|}{r_0(u^2+v^2)}\\
				&\leq \frac{pH(0,x)}{\overline{v}^2}+(1-p)(\Omega+m)+\frac{1}{r_0}=:C.
	\end{align*}
	Let us assume by contradiction that $N(R)=+\infty.$ We define the sequences $(r^{ini}_n)$ and $(r^{fin}_k)$ such that for all $k$:
	\[
		0=r^{ini}_0<r^{fin}_0<\dots <r^{ini}_k<r^{fin}_k<\dots
	\]
	and 
	\[
		\underset{k\in\mathbb{N}}{\cup}~(r^{ini}_k,r^{fin}_k)=\{r\in(0,R): |u|(r)< \tan(\overline{\theta})|v|(r)\}.
	\]
	These sequences are well-defined for all $k$ and 
	\[
		\underset{k\rightarrow+\infty}{\lim}~r^{ini}_k=R
	\]
	because lemma \ref{estimationangulaire} ensures that
	\[
		\theta'(r)\geq\Theta.
	\]
	Nevertheless, we have that for any $k>0$
	\[
		\overline{\theta}\leq\int_{r^{ini}_k}^{r^{fin}_k}\theta'(r)dr\leq C(r^{ini}_k-r^{fin}_k)
	\]
	by lemma \ref{energielevel}. This is impossible so $N(R)<+\infty.$
	\item Let us assume now that $R=+\infty$ and that there exists $r_1\geq0$ such that 
	\[H(u,v)(r_1)\leq 0 \mbox{~and~} (u,v)(r_1)\neq0.\] 
	Lemma \ref{deriveeenergie} ensures that $H(u,v)(r)<0$ for all $r>r_1$. $(u,v)$ will not cross the set $\{v=0\}$ anymore and $N(\infty) < \infty.$
	\item Let us assume next that $R=+\infty$ and that there exists $\gamma >0$ such that $H(u,v)\geq\gamma$ on $(0,R)$. Let $r_0<r_1<r_2.$ We have by lemmas \ref{deriveeenergie} and \ref{energielevel}:
	\begin{align*}
		\lefteqn{H(u,v)(r_1)-H(u,v)(r_2) = -\int_{r_1}^{r_2}\frac{d}{dr}H(u,v)(r)dr}\\
		&=\int_{r_1}^{r_2}\frac{2u^2\varphi(u,v)}{r}(p|u^2-v^2|^{p-1}+(\Omega-m))dr\\
		&\geq2(\Omega-m)\int_{r_1}^{r_2}\frac{u^2\varphi(u,v)}{r}dr.
	\end{align*}
	Let the sequences $(r^{ini}_n)$ and $(r^{fin}_k)$ be such that:
	\[
		r_0<r^{ini}_0<r^{fin}_0<\dots <r^{ini}_k<r^{fin}_k<\dots
	\]
	and 
	\[
		\underset{k\in\mathbb{N}}{\cup}~(r^{ini}_k,r^{fin}_k)=\{r>0: (u,v)(r)\in\{u>|v|>0\}\cap (\mathcal{R}^1_\epsilon)^c\}.
	\]
	These sequences are well-defined because lemma \ref{estimationangulaire} ensures that
	\[
		\theta'(r)\geq\Theta
	\]
	for any $r\geq r_0.$ We get by construction of these sequences and lemma \ref{energielevel} that
	\[
		\varphi(u,v)(r)=1 ~\text{and}~2u^2(r)\geq u^2(r)+v^2(r)\geq D_\gamma
	\]
	for all $r\in\underset{k\in\mathbb{N}}{\cup}~(r^{ini}_k,r^{fin}_k)$ so,
	\begin{align*}
		\lefteqn{H(u,v)(r_1)-H(u,v)(r_2)}\\
		&\geq  D_\gamma(\Omega-m)\underset{k\in \mathcal{A}(r_1,r_2)}{\Sigma}~\int_{r^{ini}_k}^{r^{fin}_k}\frac{dr}{r}\\
		&\geq D_\gamma(\Omega-m)\underset{k\in \mathcal{A}(r_1,r_2)}{\Sigma}~\log \left(\frac{r^{fin}_k}{r^{ini}_k}\right)
	\end{align*}
	where
	\[
		\mathcal{A}(r_1,r_2):=\{k\in\mathbb{N}:~(r^{ini}_k,r^{fin}_k)\subset (r_1,r_2)\}.
	\]
	Moreover, we have by lemma \ref{deriveeangulaire} that for any $r\geq r_0:$
	\begin{align*}
		\theta(r)'\leq \frac{pH(0,x)}{D_\gamma}+(1-p)(\Omega+m)+\frac{1}{r_0}=:c_1
	\end{align*}
	so that
	\[
		\int^{r_k^{fin}}_{r_k^{ini}}\theta'(r)dr\leq c_1(r_k^{fin}-r_k^{ini}).
	\]
		The same argument gives $c_2>0$ such that
	\[
		r^{ini}_k\leq c_2k.
	\]
	Let us remark that if $\epsilon=0,$ we would have
	\[
		\int^{r_k^{fin}}_{r_k^{ini}}\theta'(r)dr= \pi/2.
	\]
	In the regularized case, there exists a constant $\pi/2\geq c_0>0$ such that
	\[
		c_0\leq\int^{r_k^{fin}}_{r_k^{ini}}\theta'(r)dr.
	\]	
	We get then:
	\[
		\log  \left(\frac{r^{fin}_k}{r^{ini}_k}\right)\geq \log\left(1+\frac{c_0 }{c_1r^{ini}_k}\right)\geq \log\left(1+\frac{c_0 }{c_1c_2k}\right).
	\]
	Since the series $\Sigma \log\left(1+\frac{c_0 }{c_1c_2k}\right)$ diverges, there are $\epsilon_0>0$ and for all $N\in\mathbb{N}$, $M(N)\in\mathbb{N}$ such that
	\[
		M(N)>N
	\]
	and
	\[
		\underset{N\leq k\leq M(N)}{\Sigma} ~\log\left(1+\frac{c_0 }{c_1c_2k}\right)\geq \epsilon_0.
	\]
	Then, we get that:
	\begin{align*}
		\lefteqn{0=\underset{N\rightarrow+\infty}{\lim}H(u,v)(r^{ini}_N)-H(u,v)(r^{fin}_{M(N)})}\\
		&\geq D_\gamma(\Omega-m)\underset{k\in \mathcal{A}(r^{ini}_N,r^{fin}_{M(N)})}{\Sigma}~\log\left(1+\frac{c_0 }{c_1c_2k}\right)\\
		&\geq D_\gamma(\Omega-m)\underset{N\leq k\leq M(N)}{\Sigma}~\log\left(1+\frac{c_0 }{c_1c_2k}\right)\\
		&\geq D_\gamma(\Omega-m)\epsilon_0>0.
	\end{align*}
	We get the wanted contradiction.
	\item The remaining case is the one of the second point of the lemma.
\end{enumerate}
\end{proof}
%
%
\subsection{Topological results}
We are now able to give some topological properties of the $A_k$ and $I_k$ sets as in \cite{balabane2003}.
\begin{lem}\label{lem:topological}
For all $k\in \mathbb{N}$,
\begin{enumerate}[(i)]
	\item $A_k$ is an open set,\label{open}
	\item $A_k$ is bounded,\label{Abounded}
	\item $I_k$ is bounded,\label{Ibounded}
	\item $\sup A_k \in I_{k-1}\cup I_k,$ \label{supA}
	\item $\sup I_k \in I_k,$ \label{supI}
	\item if  $x \in I_k$ then there exists a neighborhood $V$ of $x$ such that  $V\subset A_k\cup I_k \cup A_{k+1}.$\label{voisinage}
\end{enumerate}
\end{lem}
The proof is slightly different from the one of \cite{balabane2003} but follows essentially their ideas.  We give it here for the sake of completeness.
\begin{proof}
Let $k\in\mathbb{N}$. Lemma \ref{continuity} ensures that $A_k$ is open. By proposition \ref{proprietedenroulement}, $I_k$ and $A_k$ are bounded. Since $\underset{n\in \mathbb{N}}{\cup}~ A_n$ is open, we easily get that $\sup A_k \notin \underset{n\in \mathbb{N}}{\cup}~ A_n$ whenever $\sup A_k$ is well-defined. Let us prove now that $x:=\sup A_k$ belongs to $\underset{n\in \mathbb{N}}{\cup}~  I_n.$  We assume that $x\notin \cup I_n$, then setting $\overline{R} > \frac{\Theta(2+k)}{\pi}+r_0,$ we have $H(u_x,v_x)(\overline{R}) >0$ because of lemma \ref{comportementinfini}. Nevertheless, there exists $y\in A_k$ as close to $x$ as we want such that $H(u_y,v_y)(\overline{R})<0$ by lemma \ref{controlenombredetour}. This contradicts the continuity of the flow of lemma \ref{continuity}. So, $x$ belongs to $\underset{n\in \mathbb{N}}{\cup}~ I_n$. Proposition \ref{trapping} ensures then point \eqref{supA}. The same arguments give point \eqref{supI}. Thanks to Proposition \ref{trapping}, we immediately get point \eqref{voisinage}. 
\end{proof}
We proved the key lemmas of \cite{balabane2003}, thus, we get the following result:
\begin{prop}
For all $\epsilon \in (0,m)$, all $k\in \mathbb{N}$, there exists a solution $(u_x,v_x) $ of (\ref{systequation:uvreg}) such that:
\begin{enumerate}[(i)]
	\item $R_x<\infty$, i.e. $(u_x,v_x)(R_x) = 0$,
	\item $N_x(0,R_x) = k.$
\end{enumerate}
\end{prop}
The proof of this proposition can be found in \cite{balabane2003}.  We give it here for the sake of completeness.
\begin{proof}
The goal of the proof is to show that $I_k\ne\emptyset$ for any $k$.  We will show this by induction on $k$. Let us remark first that $(0,E_0]\subset A_0.$ Then, $\sup A_0$ exists and belongs to $I_0$ by lemma \ref{lem:topological}. Thus, $\sup I_0$ exists and belongs to $I_0$ by point \eqref{supI} of lemma \ref{lem:topological} and $\sup A_0\leq \sup I_0$. Point \eqref{voisinage} ensures that there exists $\delta>0$ such that
	\[
		(\sup I_0-\delta,\sup I_0+\delta)\subset A_0\cup I_0\cup A_1.
	\]
	Thus, we obtain that
	\[
		(\sup I_0,\sup I_0+\delta)\subset A_1 \ne\emptyset.
	\]
	 We assume now that there is $k\in \mathbb{N}\backslash\{0\}$ and $\delta_{k-1}>0$ such that
	\[
		(\sup I_{k-1},\sup I_{k-1}+\delta_{k-1} )\subset A_k \ne\emptyset.
	\]
	Then, by lemma \ref{lem:topological}, $\sup A_k$ exists and belongs to $I_k$ since $\sup A_{k-1}\leq\sup I_{k-1}<\sup A_k$. We get also that
	\[
		\sup A_k \leq\sup I_k\in I_k.
	\]
	Then, point \eqref{voisinage} ensures that there exists $\delta_k>0$ such that
	\[
		(\sup I_k-\delta_k,\sup I_k+\delta_k)\subset A_k\cup I_k\cup A_{k+1}
	\]
	so,
	\[
		(\sup I_{k},\sup I_{k}+\delta_{k} )\subset A_{k+1} \ne\emptyset.
	\]
	We proved by induction that $A_k$ and $I_k$ are not empty.
\end{proof}

\section{Proof of the existence of localized solutions to problem \eqref{equation:stati2}}\label{section:theo}
We give here the proof of theorem \ref{resultshooting}.
\begin{proof}

Let us fix $k\in\mathbb{N}$. We write here the $\epsilon$ superscripts to emphasize the dependence of the solutions on $\epsilon$. Let $(u^\epsilon_{x_\epsilon}, v^\epsilon_{x_\epsilon})$ be a solution of (\ref{systequation:uvreg}) such that $N^\epsilon_{x_\epsilon}(R_{x_\epsilon}^\epsilon) = k$ with $\varphi_\epsilon$ defined in section \ref{subsection:regul}.  From now on, we will not write the subscript ${x_\epsilon}$ anymore for the sake of notation simplicity. We have : 
\begin{equation}\label{eq:prooftheo1}
	R^\epsilon = \rho^\epsilon(0)\leq \frac{(k+1)\pi}{\Theta}+r_0 = \overline{R}
\end{equation}
by lemma \ref{deriveeangulaire}. 
Let $\gamma>0$ such that $q:=2(1-p)(\Omega+m)C_\gamma-1>0,$ either $H(u^\epsilon, v^\epsilon)(r_0)$ is smaller than $\gamma$ or not. In that latter case, we have because of lemma \ref{controledecroissanceenergie}:
\begin{align*}
	\lefteqn{\left((R^\epsilon)^{2(\Omega+m)C_\gamma}\gamma\right)^{1-p}-\left(r_0^{2(\Omega+m)C_\gamma}H(u^\epsilon, v^\epsilon)(r_0)\right)^{1-p} }\\
 &\geq-2p(1-p)C_\gamma^p(R^\epsilon)^{q}\int_{r_0}^{R^\epsilon }| 1-2\sin^2\theta^\epsilon|^{p-1}dr.
\end{align*}
This and inequality \ref{eq:prooftheo1} give us an uniform bound on $H(u^\epsilon, v^\epsilon)(r_0)$ which does not depends on $\epsilon$. We extend now the functions $u^\epsilon$ and $v^\epsilon$ by zero on $[R^\epsilon,\overline{R}].$ By lemmas \ref{energielevel} and \ref{deriveeenergie}, we obtain a uniform bound on $(u^\epsilon, v^\epsilon)$ in $\mathcal{C}^0([r_0,\overline{R}]).$  We get then that  
$
	(u^\epsilon)^2, (v^\epsilon)^2~\text{and}~u^\epsilon v^\epsilon
$
are bounded sequences of  $W^{1,s}([r_0,\overline{R}])$ by lemma \ref{sobolev}. 

Up to the extraction, there exist  a decreasing subsequence $(\epsilon_n)$ which tends to $0,$ $U, V, W \in W^{1,s}([r_0,\overline{R}])$ such that: 
	\begin{align*}
		U_n := (u^{\epsilon_n})^2	&\underset{n\rightarrow \infty}{\longrightarrow} U\\
		V_n := (v^{\epsilon_n})^2 	&\underset{n\rightarrow \infty}{\longrightarrow} V\\
		W_n :=u^{\epsilon_n}v^{\epsilon_n} 	&\underset{n\rightarrow \infty}{\longrightarrow} W,
	\end{align*}
	in $\mathcal{C}^0([r_0,\overline{R}]).$ We can then construct a function $(\overline{u},\overline{v})$ defined on $[r_0,\overline{R}]$ which is a solution of the system of equations \eqref{equation:uv} taking care of the sign of $W$ such that $\overline{u}^2 = U$, $\overline{v}^2 = V$, $\overline{u}(r_0)>0$ and $\overline{v}(r_0)>0$. $(\overline{u},\overline{v})$ satisfies $(\overline{u},\overline{v})(\overline{R})=(0,0).$ 
	
	It remains to study the problem on $[0,r_0].$ We define  $F(x) = (u_x,v_x)(r_0),$ where $(u_x,v_x)$ is a solution of \eqref{systequation:uv}. $F$ is a one-to-one continuous function from $[E_0,\infty)$ into $F([E_0,\infty))$ where $E_0$ comes from lemma \ref{energielevel}. Let us remark that we have constructed the regularized systems so that 
	\[
		\varphi_\epsilon(u^\epsilon, v^\epsilon)(r)=1
	\]
	for all $r\in[0,r_0].$
	We have that $((u^\epsilon, v^\epsilon)(r_0))$ is a bounded sequence and
	\[
		(v^\epsilon)^2(r_0)-(u^\epsilon)^2(r_0) \geq \alpha x_\epsilon^2
	\] 
	by lemma \ref{estimationangulaire}, so $(x_\epsilon)$ is bounded. Up to another extraction, we can assume that $(x_{\epsilon_n})$ converges to $x>0.$ Since $F$ is continuous, we get that $(u_x,v_x)(r_0) = (\overline{u},\overline{v})(r_0).$ We have constructed a solution $(\overline{u},\overline{v})$ of problem \eqref{systequation:uv}.

It just remains us to show that the function obtained still have his winding number $N$ satisfying
\[
	N=N^\epsilon(R^\epsilon)=k.
\]
Let $c_+\in H^{-1}(\mathbb{R}^-_*)\cap\{ (u,v)|u<0<v\}$ and $\delta>0$ such that $B(c_+,\delta) \subset H^{-1}(\mathbb{R}^-_*)\cap\{ (u,v)|u<0<v\}.$ We write $c_- = -c_+.$ For every $\epsilon \geq0$, we join $(u^\epsilon, v^\epsilon)(r_0)$ and $0$ to define the closed curve $\gamma^\epsilon$. We define also $\gamma^0$ from $(\overline{u},\overline{v})$. Setting: 
\[N(\gamma) = -\frac{1}{2i\pi}\int_{\gamma}\left(\frac{1}{z-c_+}+\frac{1}{z+c_+}\right),\]
Lebesgue theorem shows that $N(\gamma^{\epsilon}) = k$ converge to $N({\gamma^{0}}).$ We extend now $(\overline{u},\overline{v})$ by zero and we get the result of the theorem.
\end{proof}
%
%
%
%
%
\section{The M.I.T. bag model limit}\label{section:mit}
Let $k\in\mathbb{N}.$ We denote by $(u_p,v_p)$ the solution of problem \eqref{systequation:uv} given by theorem \ref{resultshooting} which crosses $k$ times the set $\{v=0\}\backslash\{(0,0)\}$ and $R_p$ is the radius at which it hits $0$. We give here the proof of theorem \ref{theo:mitlimit}.

\begin{proof}
We fix $\eta\in(0,\overline{p}),$ $\Omega> \Omega_{\overline{p}-\eta}$ and $p\in(0,\overline{p}-\eta).$ We have by lemma \ref{controlenombredetour}:
\[R_p \leq \frac{(k+1)\pi}{\Theta}+r_0=\overline{R}.\]
Lemma \ref{estimationangulaire} ensures that $\overline{R}$ does not depend on $p$. We extend $(u_p,v_p)$ by zero on $[R_p,\overline{R}].$  We denote by $H_p$ the functions introduced in section \ref{section:preliminaryresult} to insist on the dependence on $p$.
\begin{lem} 
\label{convergencep}
There exist $g\in \mathcal{C}^0([r_0,\overline{R}])$, $(u_0,v_0)\in \mathcal{C}^0(\overline{\{|g|>0\}})$ and a decreasing sequence $(p_n)$ which converges to zero such that:
\begin{enumerate}[(i)]
	\item $(v^2_{p_n}-u^2_{p_n})$ converges uniformly to $g$ in $\mathcal{C}^0([r_0,\overline{R}]),$
	\item $(u_{p_n},v_{p_n})$ converges uniformly to $(u_0,v_0)$ on every compact interval of \newline ${\{|g|>0\}},$ $v^2_0-u^2_0 =g$ and $(u_{p_n},v_{p_n})$ is a bounded sequence of $\mathcal{C}^0([r_0,\overline{R}]),$ 
	\item $(u_0,v_0)$ is a solution of the free Dirac equation
		\[\left\{\begin{array}{rl}
			u'+\frac{2u}{r}=&v(\Omega-m)\\
			v'=&-u(\Omega+m)
		\end{array}\right.\] 
		on ${\{|g|>0\}}$.
\end{enumerate}
\end{lem}
\begin{proof}
Let $\gamma>0$. The arguments of the proof of theorem \ref{resultshooting} ensure that the sequence $(H_p(u_p^2,v^2_p))_p$ is bounded on $[r_0,\overline{R}]$ uniformly in $p$. We claim that $(u_p,v_p)_p$ is bounded on $[r_0,\overline{R}]$ uniformly in $p$ too. Let us assume by contradiction that $(u_{p},v_{p})$ is not bounded. Up to a subsequence, there exists $(r_{p_n})_n\in [r_0,\overline{R}]^{\mathbb{N}}$ such that $(u_{p_n},v_{p_n})(r_{p_n})=:(u_n,v_n)$ satisfies\begin{equation*}
\begin{array}{l}
(\Omega+m)u_n^2 + (\Omega-m)v_n^2\underset{n\rightarrow\infty}{\longrightarrow}\infty,\\
H_{p_n}(u_n,v_n)\leq C,
\end{array}
\end{equation*}
for some constant $C>0$, so that 
\[(v^2_n-u^2_n)|v^2_n-u^2_n|^{p_n-1}\underset{n\rightarrow\infty}{\longrightarrow}\infty.\]
$v^2_n-u^2_n$ is then nonnegative for $n$ big enough and $v^2_n-u^2_n\underset{n\rightarrow\infty}{\longrightarrow}\infty,$ thus 
\[H_{p_n}(u_n,v_n)=-\frac{1}{2}|v^2_n-u^2_n|^{p_n}+ \frac{\Omega-m}{2}(v^2_n-u^2_n)+\Omega u^2_n\underset{n\rightarrow\infty}{\longrightarrow}\infty.\]
This is the wanted contradiction. Moreover, we have 
\[\frac{d}{dr}(v^2_{p}-u^2_{p})(r) = 4\left(\frac{u^2_{p}}{r}-\Omega u_{p}v_{p}\right),\]
so that $v^2_{p}-u^2_{p}$ is equicontinuous and bounded on $[r_0,\overline{R}].$ Ascoli's theorem shows the first point. On every compact interval of $\{|g|>0\},$ $(u_{p_n},v_{p_n})$ is also equicontinuous and bounded. Ascoli's theorem gives us the second one. The remaining is immediate.  
\end{proof}
Let $R^i_{p}$ be the $i$-th radius at which  $(u_p,v_p)$ crosses the set $\{|u|=|v|\}$, $\widetilde{R}^i_p$ the $i$-th radius at which  $(u_p,v_p)$ crosses  the set $\{uv=0\}$, where $i$ belongs to $\{1,\dots,2k\}.$ Up to extraction, there exist $(R^i_0)_i,$ $(\widetilde{R}^i_0)_i$ such that:
\[\begin{array}{l}R^i_{p_n}\underset{n\rightarrow\infty}{\longrightarrow}R^i_0 ~\mbox{and}~g(R^i_0)=0,\\
		\widetilde{R}^i_{p_n}\underset{n\rightarrow\infty}{\longrightarrow}\widetilde{R}^i_0,
\end{array}\]
and $r_0\leq R^1_0\leq \widetilde{R}^1_0\leq R^2_0\leq \dots \leq \widetilde{R}^{2k}_0.$
\begin{lem}
For all $i$ even, $\emptyset \ne(R^i_0,R^{i+1}_0)\subset\{g>0\}$ and $\emptyset \ne(r_0,R^{1}_0)\subset\{g>0\}$.
\end{lem}
\begin{proof}
We recall that  
\[
	E_0^p=(\Omega-m)^\frac{1}{2(1-p)}= \sup\{v| \exists u; H_p(u,v)=0\},
\]
 and $\underset{p\rightarrow0}{\lim}~E_0^p = \sqrt{\Omega-m}$. Thus, we obtain for all $i$ even, 
\[
	(v^2_{p_n}-u^2_{p_n})(\widetilde{R}^i_{p_n}) \geq (\Omega-m)/2
\] 
for $n$ big enough, so that 
\[
	g(\widetilde{R}^i_0)\geq (\Omega-m)/2.
\] 
This ensures that $R^i_0<\widetilde{R}^i_0< R^{i+1}_0$. We claim that $(R^i_0,R^{i+1}_0)\subset\{g>0\}. $ Let  $r^+\in(\widetilde{R}^i_0,R^{i+1}_0)$, we have
\begin{align*}
	&\theta_{p_n}(R^{i+1}_{p_n})-\theta_{p_n}(r^+)\geq\Theta(R^{i+1}_{p_n}-r^+),\\
	&\theta_{p_n}(r^+)-\theta_{p_n}(\widetilde{R}^i_0)\geq\Theta(r^+-\widetilde{R}^i_0),
\end{align*}
so
\begin{align}
	\label{lem:limangle1}&\underset{n\rightarrow+\infty}{\lim\sup}~\theta_{p_n}(r^+)<\underset{n\rightarrow+\infty}{\lim}~\theta_{p_n}(R^{i+1}_{p_n})\\
	\label{lem:limangle2}&\underset{n\rightarrow+\infty}{\lim\inf}~\theta_{p_n}(r^+)>\underset{n\rightarrow+\infty}{\lim}~\theta_{p_n}(\widetilde{R}^i_0).
	\end{align}
We also have
\[
	H_{p_n}(u_{p_n},u_{p_n})(r_+)\geq0
\]
so that by Point \eqref{lem:point6} of Lemma \ref{energielevel},
\[
	|v_{p_n}(r_+)|\geq C
\]
for some positive content $C$.
 We have by lemma \ref{convergencep} that  $(u_{p_n},v_{p_n})$ is a bounded sequence of $\mathcal{C}^0([r_0,\overline{R}]),$ and 
 \[
 	\underset{n\rightarrow\infty}{\lim}~v^2_{p_n}(r^+)-u^2_{p_n}(r^+) = g(r^+).
\]
Assume by the contradiction that $g(r_+)=0,$ then, up to extraction, we have
  \[
 	\underset{n\rightarrow\infty}{\lim}~v^2_{p_n}(r^+)=\underset{n\rightarrow\infty}{\lim}~u^2_{p_n}(r^+)\geq C^2
\] 
but, this is in contradiction with inequalities \eqref{lem:limangle1} and \eqref{lem:limangle2}. Thus, we get that
\[
	(\widetilde{R}^i_0,R^{i+1}_0)\subset\{g>0\}. 
\] 
The same argument works as well for $r^-\in (R^{i}_0,\widetilde{R}^i_0)$ and $(r_0,R^1_0).$ This gives us the lemma. 
\end{proof}
\begin{rem}
The  limiting function is more complicated to tackle on the intervals $(R^i_0,R^{i+1}_0)$ when $i$ is odd and different behaviors may occur when $H_{p_n}(R^{i+1}_{p_n})\geq 1/2$ or $1/2\geq H_{p_n}(R^{i}_{p_n})\geq 0.$
\end{rem}

Let us now consider the following mapping:
\[\begin{array}{cl}\Gamma: &(-\overline{p}/2,\overline{p}/2)\times(\sqrt{\Omega-m}/2,\infty)\rightarrow \mathbb{R}^3\\
	&(p,x)\longmapsto (p,u_{p,x}(r_0),v_{p,x}(r_0)),
\end{array}
\]
where $(u_{p,x},v_{p,x})$ is the solution of the problem (\ref{systequation:uv}$)_p$ with $x$ as initial condition. $\Gamma$ is an injective continuous map. We denote by $(x_n)$ the sequence of initial conditions related with $(u_{p_n},v_{p_n})$. We get by lemma \ref{estimationangulaire} that
\[
	v^2_{p_n}(r_0)\geq v^2_{p_n}(r_0)-u^2_{p_n}(r_0)\geq \alpha x_n^2
\]
so that, the sequence $(x_n)$ is bounded by lemma \ref{convergencep}. Up to extraction, we have
\[
	\underset{n\rightarrow+\infty}{\lim}~x_n=x\in \mathbb{R}_+.
\]
 The continuity of $\Gamma$ ensures, 
\[
	\underset{n\rightarrow+\infty}{\lim}~\Gamma(p_n,x_n)=\Gamma(0,x)
\]
and $x$ satisfies
\[
	(u_{0,x},v_{0,x})(r_0) = (u_0,v_0)(r_0)
\]
where $(u_0,v_0)$ comes from lemma \ref{convergencep}. $(u_{p_n},v_{p_n})$ converges to $(u_{0,x},v_{0,x})$ uniformly on $[0,r_0].$ Thus we get the theorem.
\end{proof}

\appendix
\section*{Appendix}
We give here the proof of lemma \ref{energielevel}. The proof of the first four points is given in \cite{balabane1990}.
\begin{proof}
We have 
\[
	\underset{u^2+v^2 \rightarrow +\infty}{\lim\inf}~\frac{H(u,v)}{u^2+v^2}=\frac{\Omega-m}{2}
\] 
and this gives us the first two points. The fourth point is immediate. The proof of  the third one is straightforward and can be found in \cite{balabane1990}. We will now prove the fifth point. We denote
\[
	\overline{C}_{\gamma} = \sup\{U>0~|~ \exists v,~ H(\sqrt{U},v) = \gamma\}/\gamma
\]
for all $\gamma>0$. We have by definition,  for all $(u,v) \in \mathbb{R}^2$:
\[
	H(u,v) = \gamma \Rightarrow \overline{C}_{\gamma}H(u,v)\geq u^2.
\]
Since $H^{-1}(\{\gamma\})$ is compact, there exists $(u_0,v_0)\in H^{-1}(\{\gamma\})$ such that  $\gamma \overline{C}_{\gamma} = u_0^2$. Thanks to the symmetries of $H$, we can assume that $u_0,v_0\geq 0.$ We denote now 
\[
	E_\gamma:= \sup\{v:~\exists u, ~H(u,v)=\gamma\}.
\]
By the implicit function theorem, there exists a regular function 
\[
	u:v\in [0,E_\gamma)\mapsto u(v)\in\mathbb{R}_+
\] 
such that:
\[
	\{(u,v):~u,v>0,~H(u,v)=\gamma\}=\{(u(v),v):~v\in(0,E_\gamma)\}
\] 
and
\[
	\frac{d}{dv}u(v) = \frac{v[p|v^2-u^2|^{p-1}-(\Omega-m)]}{u[p|v^2-u^2|^{p-1}+(\Omega+m)]}.
\] 
$v\in(0,E_\gamma)$ satisfies $\frac{d}{dv}u(v)>0$ if and only if
\[
	|v^2-u^2|<\left(\frac{\Omega-m}{p}\right)^{\frac{1}{p-1}}
\]
so that, the function $v\mapsto u(v)$ has at most two local maxima  in $0$ and in $v_1$ defined by
\[
	v_1^2-u(v_1)^2=\left(\frac{\Omega-m}{p}\right)^{\frac{1}{p-1}}.
\]
We get 
\[
	\gamma= H(u(0),0)=\frac{1}{2}\left(u(0)^{2p}+(\Omega+m)u(0)^2\right)
\]
and 
\[
	\gamma=H(u(v_1),v_1) = -\frac{1-p}{2}\left(\frac{\Omega-m}{p}\right)^{\frac{p}{p-1}}+\Omega u(v_1)^2.
\]
We define now
\[
	C_\gamma^0 = u(0)^2/\gamma=\frac{2}{\frac{1}{u(0)^{2(1-p)}}+\Omega+m}
\]
and
\[
	C_\gamma^1=u(v_1)^2/\gamma= \frac{1}{\Omega}\left(1+\frac{(1-p)\left(\frac{\Omega-m}{p}\right)^{\frac{p}{p-1}}}{2\gamma}\right).
\]
It is straightforward to see that $\gamma\mapsto C^0_\gamma$ is a non-decreasing function whereas $\gamma\mapsto C^1_\gamma$ is non-increasing and 
\[
	\overline{C}_\gamma =\max\{C^0_\gamma,C^1_\gamma\}.
\]
We have that
\[
	\underset{\gamma\rightarrow +\infty}{\lim}~C^0_\gamma=\frac{2}{\Omega+m}
\]
so that defining
\[
	C_\gamma=\max\{\overline{C}_\gamma,\frac{2}{\Omega+m}\},
\]
we get that $\gamma\mapsto C_\gamma$ is a non-increasing function such that for every $\gamma >0$, $(u,v)\in H^{-1}([\gamma,+\infty)),$
\[
	C_{H(u,v)}H(u,v)\geq C_\gamma H(u,v)\geq u^2
\]
and
\[
	\underset{\gamma\rightarrow 0}{\lim}~C_\gamma=+\infty.
\]
Let us remark now that $H^{-1}((-\infty,\gamma))$ is a bounded open set for all $\gamma>0$, so we can define:
\[
	D_\gamma=\sup\{D>0:~B(0,\sqrt{D})\subset H^{-1}((-\infty,\gamma))\}\in(0,+\infty)
\]
where $B(0,r)$ is the euclidean ball of $\mathbb{R}^2$ of radius $r$. We immediately get that if $(u,v)$ satisfies $H(u,v)\geq \gamma>0$ then $(u,v)\notin B(0,\sqrt{D_\gamma})$ and this is the result.

Let us now prove the last point. Just as in the proof of the previous point, we can define thanks to the implicit function theorem, a regular function 
\[
	u:v\in(0,E_0)\mapsto u(v)\in\mathbb{R}^*_+
\]
such that
\[
	\{(u,v):~H(u,v)=0,~0<u,v\}=\{(u(v),v):~u\in(0,E_0)\}.
\]
This function is increasing on $(0,\overline{v})$ for 
\[
	\overline{v}^2=\left(\frac{p}{\Omega-m}\right)^{\frac{1}{1-p}}+\frac{1-p}{2\Omega}\left(\frac{p}{\Omega-m}\right)^{\frac{p}{1-p}}
\]
and decreasing on $(\overline{v},E_0).$ We define $\overline{\theta}\in(0,\pi/4)$ by
\[
	\tan(\overline{\theta})	=\frac{u(\overline{v})}{\overline{v}}.	
\] 
Let us define now
\[
	\Gamma:v\in(0,E_0]\mapsto\frac{u(v)}{v}\in[0,1).
\]
It is straightforward to see that the function $\Gamma$ is decreasing from $[\overline{v},E_0]$ in $[0,\overline{\theta}],$ one-to-one and onto. For any $\alpha\in[0,\overline{\theta}],$ the function
\[
	v\in\mathbb{R}^*_+\mapsto H(\tan(\alpha) v,v)
\] 
is strictly convexe,
\[
	\underset{v\rightarrow+\infty}{\lim}~H(\tan(\alpha) v,v)=+\infty~\text{and}~H(\tan(\alpha) v,v)\underset{v\rightarrow 0}{\sim}~-\frac{v^{2p}(1-\tan(\alpha)^2)^p}{2}
\]
so that, there is a unique $v_\alpha>0$ such that $H(\tan(\alpha) v_\alpha,v_\alpha)=0.$ We have $\Gamma(v_\alpha)=\tan(\alpha)$ so $v_\alpha\geq\overline{v}.$ We get also that if $(u,v)$ satisfies $H(u,v)\geq0$ and $|u|\leq \tan(\overline{\theta})|v|$ then there is a unique $\alpha\in[0,\overline{\theta}]$ such that $\tan(\alpha)=\frac{u}{v}$ and $v\geq v_\alpha\geq \overline{v}.$
Since 
\[
	\underset{p\rightarrow 0}{\lim}~\overline{v}^2=\frac{1}{2\Omega}
\]
and
\[
	\underset{p\rightarrow 0}{\lim}~\overline{v}^2-u(\overline{v})^2=0,
\]
we can choose smaller constants for $\overline{\theta}$ and $\overline{v}$ that do not depend on $p$.

\end{proof}
We prove now lemma \ref{estimationangulaire}.
\begin{proof}
We denote by $(u_{x},v_{x})$ the solution of \eqref{systequation:uv}$_p$. We begin as in \cite{balabane1990}. Let $r_0> \frac{1}{\Omega},$ 
\[
	\overline{R}_x = \sup\{r>0|v_x>|u_x|\}\in(0,+\infty]
\] 
and $S_x = \min(\overline{R}_x, r_0).$ For $r\in(0,S_x)$, we have:
\[\begin{array}{ll}
\frac{d}{dr}(v^2_x-u^2_x) 	&= 2(v'_xv_x-u'_xu_x) \\
					&= 4(\frac{u^2_x}{r}-\Omega u_x v_x) \\
					&\geq 4\Omega(u_x^2-v_x^2)-4(\Omega-\frac{1}{r_0}) u^2_x \\
					&\geq 4\Omega(u_x^2-v_x^2)-4(\Omega-\frac{1}{r_0}) x^2
\end{array}\]
because $S_x\leq r_0,\overline{R}_x.$
We get: 
\[\frac{d}{dr}(e^{4\Omega r}(v^2_x-u^2_x))+4(\Omega-\frac{1}{r_0}) x^2e^{4\Omega r} \geq 0 ~\mbox{on} ~(0,S_x).\]
and
\begin{align*}
v^2_x-u^2_x 	&\geq x^2\left(e^{-4\Omega r_0}(1+\frac{(\Omega-\frac{1}{r_0})}{\Omega})-\frac{(\Omega-\frac{1}{r_0})}{\Omega}\right)~\mbox{on} ~[0,S_x).
\end{align*}
We want to show that we can choose $r_0>\frac{1}{(1-p)(\Omega-m)}>\frac{1}{\Omega}$.
 We define  
 \[
 	g:(m,\infty)\times(0,1)\rightarrow \mathbb{R}
\]
by
\[\begin{array}{ll}
	g(\Omega,p)= \exp\left(-\frac{4\Omega}{(\Omega-m)(1-p)}\right)\left(1+\frac{p\Omega+(1-p)m}{\Omega}\right)-\frac{p\Omega+(1-p)m}{\Omega},
\end{array}\]
and 
\[f: p\in(0,1) \mapsto e^{-\frac{4}{1-p}}(1+p)-p\in\mathbb{R}.\] 
On one hand, for $p$ fixed, $\Omega\mapsto g(\Omega,p)$ is increasing and
\[\lim_{\Omega \rightarrow \infty}g(\Omega,p) = f(p), ~~ \lim_{\Omega \rightarrow m}g(\Omega,p) =-1.\]
On the other hand, $f$ is decreasing and \[\lim_{p\rightarrow 0}f(p) = e^{-4}>0, ~~ \lim_{p\rightarrow 1}f(p) = -1.\]
Thus, there exists a unique $\overline{p}\in (0,1)$ such that 
\[
	\forall p\in(0,\overline{p}),~ f(p)>f(\overline{p})=0,
\]
 and for  $p\in(0,\overline{p})$, a unique $\Omega_p>m$ such that 
 \[
 	\forall \Omega> \Omega_p,~ g(\Omega,p)>g(\Omega_p,p)=0.
\]
Finally, for $0<p<q<\overline{p},$ we have for all $\Omega >\Omega_q>m,$
\[ \frac{1}{(\Omega-m)(1-p)}<\frac{1}{(\Omega-m)(1-q)}=:r_0\]
$\alpha := g(\Omega,q)>g(\Omega_q,q)=0.$ Then, we get
\[v_{x,p}^2-u_{x,p}^2 \geq \alpha x^2\mbox{~ for all~} r\in[0,S_x).\]
This ensures that $S_x = r_0$ and that the first two points of lemma \ref{estimationangulaire} are true. The latter one is an easy consequence of lemma \ref{deriveeangulaire}.
\end{proof}


\subsection*{Acknowledgment}
This problem has been proposed by Patricio Felmer. The author would like to thank Patricio Felmer and Eric S\'{e}r\'{e} for useful discussions and helpful comments. This work was partially supported by the Grant ANR-10-BLAN 0101 of the French Ministry of research.

\bibliography{bibliographiebibdesk}
\bibliographystyle{plain}

\end{document}